\documentclass[a4paper,leqno,12pt]{amsart}
\usepackage{amsmath,amsthm}
\usepackage{amssymb}
\usepackage{xspace}
\usepackage{bm}
\usepackage{amstext}
\usepackage{amsfonts}
\usepackage{graphicx}
\usepackage[mathscr]{euscript}
\usepackage{amscd}
\usepackage{latexsym}
\usepackage{enumerate}
\usepackage{mathrsfs}
\usepackage[cmtip,all]{xy}
\usepackage{xcolor}
\usepackage{standalone}
\usepackage{tikz}
\usetikzlibrary{arrows.meta}
\usepackage{tikz-cd}
\usetikzlibrary{intersections}
\usetikzlibrary{calc}
\usetikzlibrary{decorations.markings}
\usetikzlibrary{decorations.pathreplacing}
\usetikzlibrary{arrows}
\usetikzlibrary{shapes,arrows,shapes.multipart}
\usepackage[colorlinks=true, citecolor=blue, linkcolor=blue]{hyperref}

\newtheorem{theorem}{Theorem}[section]
\newtheorem{lemma}[theorem]{Lemma}
\newtheorem{proposition}[theorem]{Proposition}
\newtheorem{corollary}[theorem]{Corollary}
\theoremstyle{definition}
\newtheorem{definition}[theorem]{Definition}
\newtheorem{example}[theorem]{Example}
\newtheorem{remark}[theorem]{Remark}
\newtheorem*{thma}{Theorem A}
\newtheorem*{thmb}{Theorem B}
\newtheorem*{thmc}{Theorem C}
\newtheorem*{thmd}{Theorem D}
\newtheorem*{thme}{Theorem E}

\newcommand{\F}{\mathbb{F}}
\newcommand{\Sub}{\operatorname{Sub}}
\newcommand{\AGL}{\operatorname{AGL}}
\newcommand{\SD}{\operatorname{SD}}

\newcommand{\cI}{\mathcal{I}}
\newcommand{\cT}{\operatorname{Tr}}
\newcommand{\Cat}{\operatorname{Cat}}

\tikzset{
  hasseedge/.style = {draw=gray!55, thin},
  solidarrow/.style = {->, thick},
  dashedarrow/.style = {->, dashed, thick},
  crossarrow/.style  = {->, dashed, red!70!black, thick},
  mynode/.style = {circle, draw, inner sep=2pt, minimum size=16pt, font=\small}
}

\begin{document}

\title{Minimal generating sets of transfer systems for more non-Abelian Groups}

\author{Bheemarasetty Chakravarthy}
\address{Department of Mathematics, Indian Institute of Technology Roorkee,
  Roorkee 247667, India}
\email{b\_chakravarthy@ma.iitr.ac.in}

\author{Surojit Ghosh}
\address{Department of Mathematics, Indian Institute of Technology Roorkee,
  Roorkee 247667, India}
\email{surojit.ghosh@ma.iitr.ac.in}


\begin{abstract}
For a finite group $G$, $N_\infty$ operads encode collections of norm maps,
and by work of Blumberg--Hill and Rubin their homotopy category is equivalent
to the poset of $G$--transfer systems on the subgroup lattice of $G$.
In \cite{ABB+25} the authors defined the \emph{width} $w(G)$ as the minimal
size of a generating set for the complete $G$--transfer system and identified
it with the number of conjugacy classes of proper meet irreducible subgroups
of $G$, and the \emph{complexity} $c(G)$ as the maximum, over all transfer
systems $T$, of the size of a minimal generating set for $T$.

We compute $w(G)$ for the semidihedral groups $\SD_{2^n}$ ($n\ge 4$) and
the affine Frobenius groups $\AGL(1,p^n)\cong \mathbb{F}_{p^n}\rtimes
\mathbb{F}_{p^n}^\times$, extending existing calculations and highlighting how
subgroup lattice structure governs equivariant multiplicative complexity.
We also compute $c(D_{p^n})$ for dihedral groups of order $2p^n$ with $p$
an odd prime, establishing $c(D_{p^n})=\lfloor 3n/2\rfloor+1$, and derive
the lower bound $c(\SD_{2^n})\ge\lfloor 5(n-1)/2\rfloor$.
\end{abstract}

\maketitle

\section{Introduction}

Equivariant homotopy theory has undergone a remarkable conceptual expansion
over the past decade, driven in large part by a deeper understanding of
multiplicative structures in genuine $G$--spectra.
In contrast with the nonequivariant setting, commutativity in equivariant
stable homotopy theory is not a single notion but a hierarchy of structures
governed by the existence of norm maps which encode multiplicative induction
along subgroup inclusions.
These norms play a central role in modern developments, including the solution
of the Kervaire invariant problem and subsequent advances in equivariant
algebraic topology.

A systematic framework for organizing such multiplicative structures was
introduced by Blumberg and Hill in \cite{BH15} through the theory of
\(N_\infty\) operads.
An \(N_\infty\) operad is a genuine equivariant operad whose algebras admit
a specified collection of norm maps, interpolating between naive
$E_\infty$--ring $G$--spectra (with no nontrivial norms) and fully genuine
commutative ring $G$--spectra (with all norms present).
A central insight of \cite{BH15} is that the homotopy theory of \(N_\infty\)
operads is governed by purely combinatorial data, called \emph{indexing
systems}, consisting of certain collections of finite $H$--sets closed under
natural operations.

Rubin subsequently recast this structure in a strikingly economical form.
In \cite{Rub21a,Rub21b}, he introduced the notion of a \emph{$G$--transfer
system} -- a partial order on the subgroup lattice of a finite group $G$
satisfying axioms reflecting the formal properties of norms.
He proved that the homotopy category of \(N_\infty\) operads is equivalent
to the poset of $G$--transfer systems.
Thus the intricate homotopy--theoretic problem of classifying multiplicative
equivariant structures becomes a problem in lattice theory: understanding
certain compatible refinements of the subgroup lattice.
Bonventre and Pereira \cite{BP21} developed a parallel theory of genuine
equivariant operads, and Balchin, Barnes, and Roitzheim \cite{BBR21}
discovered surprising links with Catalan numbers and associahedra: the number
of transfer systems on a chain of length \(n\) equals the \((n+1)\)-st
Catalan number \(\Cat(n+1)\).

In a recent article, Adamyk, Balchin, Barrero, Scheirer, Wisdom, and Zapata
Castro \cite{ABB+25} introduced two numerical invariants attached to a group
\(G\) in this context: the \emph{width} and the \emph{complexity}.
The width \(w(G)\) is defined as the size of a minimal generating set for the
complete \(G\)-transfer system (the one containing every possible edge
$K\to H$ with $K\le H$).
A key result of \cite{ABB+25} identifies \(w(G)\) with the number of
conjugacy classes of proper \emph{meet irreducible subgroups} of \(G\).
The complexity $c(G)$ is the companion invariant defined as the maximum
cardinality of a minimal generating set over all transfer systems in
$\cT(G)$; it resists a direct lattice-theoretic description and requires
careful combinatorial work.

Widths have now been computed for several important families of groups.
In particular, Klanderman, Lewis, Monson, Shibata, and Van Niel \cite{KLM+25}
determined $w(G)$ for dihedral, quaternion, and dicyclic groups, revealing
systematic patterns tied to the structure of their subgroup lattices.
These calculations demonstrate that width reflects subtle interactions between
normal subgroups, maximal subgroups, and chains in $\Sub(G)$.
The first exact complexity computation appears in \cite{ABB+25}, where it is
shown that
\[
  c(C_{p^n q}) = \begin{cases} 3k+1 & n=2k,\\ 3k+2 & n=2k+1,\end{cases}
\]
for distinct primes $p$ and $q$.

The purpose of the present paper is to extend this program to three further
families of groups whose subgroup structures exhibit distinct features: the
semidihedral $2$--groups, the affine linear groups over finite fields and Modular maximal-cyclic groups.
We additionally compute the complexity for the dihedral groups $D_{p^n}$,
whose subgroup lattice presents a genuinely richer structure than the cyclic
case, requiring a new bridge-exclusion argument.

\begin{thma}\label{thm:sd}
For the Semidihedral group \(\SD_{2^n}\) with \(n\ge 4\), one has
\(w(\SD_{2^n}) = 2n-2\).
\end{thma}
See Theorem~\ref{tsemid}.

\begin{thmb}\label{thm:agl}
For the Frobenius group \(G = \AGL(1,p^n) \cong \mathbb{F}_{p^n} \rtimes
\mathbb{F}_{p^n}^\times\), where \(p\) is prime and \(n\ge 1\),
\[
  w(G) = \Omega(p^n-1) + \tau(n),
\]
with \(\Omega(m)\) the total number of prime factors of \(m\) (counted with
multiplicity) and \(\tau(m)\) the number of positive divisors of \(m\).
\end{thmb}
See Theorem~\ref{tfrob}.

\begin{thmc}\label{thm:mod}
For the Modular maximal-cyclic group \(M_n(2)\) with \(n\ge 4\), one has
\(w(M_n(2)) = 2(n-1)\).
\end{thmc}
See Theorem~\ref{tmod}.

\begin{thmd}\label{thm:complex}
Let $p$ be an odd prime and $n\ge 1$.  Then
\[
  c(D_{p^n}) = \left\lfloor\frac{3n}{2}\right\rfloor + 1
  = \begin{cases} 3k+1 & n=2k,\\ 3k+2 & n=2k+1.\end{cases}
\]
In particular, the sequence $c(D_{p^n})$ as $n$ ranges over positive integers
is precisely the sequence of positive integers not divisible by~$3$.
\end{thmd}
See Theorem~\ref{tcomplex}.

\begin{thme}\label{thm:sd_lower}
For the semidihedral group $\SD_{2^n}$ with $n\ge 4$,
\[
  c(\SD_{2^n})\ge\left\lfloor\frac{5(n-1)}{2}\right\rfloor
  =\begin{cases}5m-3&n=2m,\\5m&n=2m+1.\end{cases}
\]
\end{thme}
See Theorem~\ref{tsdlower}.

\vspace{0.5cm}

{\bf Organization.}
Section~2 recalls the necessary background on transfer systems, meet
irreducible subgroups, width, and complexity.
Section~3 contains the proofs of Theorems~\ref{tsemid} and~\ref{tfrob}.
Section~4 establishes Theorem~\ref{tcomplex}, computing the complexity of
$D_{p^n}$.
Section~5 proves Theorem~\ref{tsdlower}, the lower bound for $c(\SD_{2^n})$.

\section{Background}

Throughout this paper, \(G\) denotes a finite group.
We write \(\Sub(G)\) for the set of all subgroups of \(G\), partially ordered
by inclusion.
Equipped with intersection as meet and subgroup generation as join,
\(\Sub(G)\) forms a finite lattice.
Many of the structures arising in equivariant homotopy theory are controlled
by additional order-theoretic data imposed on this lattice.

\subsection{Transfer systems}

Transfer systems were introduced as a combinatorial model for the homotopy
theory of \(N_\infty\)-operads.
They encode, in purely order-theoretic terms, which transfer (or norm) maps
are present between subgroups.

\begin{definition}
A \emph{\(G\)-transfer system} is a partial order \(\to\) on \(\Sub(G)\)
such that:
\begin{enumerate}[(i)]
\item \(K\to H\) implies \(K\le H\);
\item \(H\to H\) for every \(H\le G\);
\item \(L\to K\) and \(K\to H\) imply \(L\to H\);
\item \(K\to H\) and \(L\le H\) imply \((K\cap L)\to (H\cap L)\);
\item \(K\to H\) implies \(gKg^{-1}\to gHg^{-1}\) for all \(g\in G\).
\end{enumerate}
Reflexive edges are customarily omitted from diagrams.
\end{definition}

Condition~(i) ensures that transfers only occur along subgroup inclusions.
Conditions~(ii) and~(iii) state that $\to$ is reflexive and transitive,
hence a refinement of the inclusion order.
Condition~(iv) encodes compatibility with restriction along subgroups and
reflects the double-coset behavior of transfers in equivariant homotopy
theory.
Finally, condition~(v) enforces conjugation invariance, since transfer data
depends only on subgroups up to conjugacy.

A transfer system may therefore be viewed as a choice of certain subgroup
inclusions along which transfer maps are permitted, subject to natural closure
conditions.
The collection of all \(G\)-transfer systems forms a poset under inclusion of
relations.

\subsection{Meet irreducible subgroups}

The structure of transfer systems is governed by the lattice-theoretic
properties of \(\Sub(G)\), and in particular by certain distinguished
subgroups.

\begin{definition}
A proper subgroup \(K<G\) is \emph{meet irreducible} if it cannot be written
as the intersection of two strictly larger subgroups.
Equivalently, \(K\) possesses a unique minimal overgroup (the smallest
subgroup properly containing \(K\)).
\end{definition}

\begin{definition}[{\cite[Definition 4.1]{ABB+25}}]
The \emph{width} \(w(G)\) is the cardinality of a minimal generating set of
the complete \(G\)-transfer system (the transfer system containing all edges
\(K\to H\) with \(K\le H\)).
\end{definition}

Here, a generating set is a collection of relations whose closure under
axioms~(ii)--(v) yields the complete transfer system.
In other words, starting from these chosen edges and repeatedly applying
transitivity, restriction via intersections, and conjugation, one obtains
every relation \(K\to H\) with \(K\le H\).
The width therefore measures the minimal number of primitive transfer
relations required to force all possible transfers.

The fundamental connection between width and meet irreducible subgroups is
given by the following result.

\begin{proposition}[{\cite[Proposition 4.5]{ABB+25}}]\label{prop:width}
For any finite group \(G\), \(w(G)\) equals the number of conjugacy classes
of proper meet irreducible subgroups of \(G\).
Moreover, the set $\{[H]\to G \mid H<G \text{ meet irreducible}\}$ is a
minimal generating set for the complete transfer system.
\end{proposition}

Thus computing \(w(G)\) reduces to classifying conjugacy classes of meet
irreducible subgroups.

\subsection{Complexity and partial rainbows}

The complexity invariant is a companion to the width, recording the hardest
transfer system to generate rather than the easiest.

\begin{definition}[{\cite{ABB+25}}]
If $m(T)$ denotes the size of any minimal generating set for $T\in\cT(G)$
(well-defined by \cite[Corollary~2.14]{ABB+25}), the \emph{complexity} of
$G$ is
\[
  c(G) := \max\bigl\{m(T)\mid T\in\cT(G)\bigr\}.
\]
\end{definition}

A key tool for producing lower bounds on complexity is the notion of a
partial rainbow.

\begin{definition}[{\cite[Definitions 3.1--3.2]{ABB+25}}]
A \emph{rainbow} on $\Sub(G)$ is a collection of $P$-arcs
$(j_1,k_1),\dots,(j_r,k_r)$ that are pairwise nested (i.e., strictly
ordered by inclusion as intervals).
A \emph{partial rainbow} is a set $S$ of arrows in $I(\Sub(G))$ whose
$P$-arcs form a rainbow, with at most one arrow chosen from each conjugacy
class.
\end{definition}

By \cite[Proposition~3.3]{ABB+25}, every partial rainbow is a minimal
generating set for the transfer system it generates, so any partial rainbow
of size $s$ witnesses $c(G)\ge s$.

\section{Width computations}

We now compute the width invariant for the families of groups appearing in
the main results, starting with the semidihedral $2$--groups.

\subsection{Semidihedral groups \(\SD_{2^n}\)}

We now determine the width of the semidihedral $2$--groups.
Throughout this subsection fix an integer $n\ge 4$ and let
\[
  G = \SD_{2^n}
  = \langle a,b \mid a^{2^{n-1}}=b^2=1,\; bab=a^{2^{n-2}-1}\rangle.
\]
Set $N=2^{n-1}$ and $k=N/2-1=2^{n-2}-1$.
A direct computation shows $k^2\equiv 1\pmod{N}$.
Conjugation by $a$ acts on $b$ via
\[
  a^r b a^{-r} = b a^{r(k-1)},
\]
and since $k-1=N/2-2$ is even, the exponent $r(k-1)$ is divisible by $2$,
so $a^{r(k-1)}\in\langle a^2\rangle$.
The centre of $G$ is
\[
  Z(G) = \{1,a^{N/2}\} =: \langle z\rangle,
\]
which has order two.

\begin{theorem}\label{tsemid}
For $n\ge 4$ one has $w(\SD_{2^n})=2n-2$.
\end{theorem}

\begin{proof}
By Proposition~4.5 of \cite{ABB+25}, the width equals the number of
conjugacy classes of proper meet irreducible subgroups of $G$.
We determine these subgroups explicitly.

The group $\SD_{2^n}$ possesses three maximal subgroups of index two, each
normal:
\[
  M_c = \langle a\rangle\cong C_{2^{n-1}},\quad
  M_d = \langle a^2,b\rangle\cong D_{2^{n-1}},\quad
  M_q = \langle a^2,ab\rangle\cong Q_{2^{n-1}}.
\]
Their pairwise intersections equal \(\langle a^2\rangle\cong C_{2^{n-2}}\).
Subgroups contained in two distinct maximals lie in this intersection and
therefore cannot be meet irreducible; consequently any proper meet irreducible
subgroup is contained in exactly one maximal.

The maximal \(M_c\) itself is meet irreducible, its unique cover being \(G\);
no proper subgroup of \(M_c\) qualifies.

Now consider \(M_d\).
Its rotation subgroup \(\langle a^2\rangle\) has order \(2^{n-2}\); the
reflections are the elements \(x_s = a^{2s}b\) (\(0\le s<2^{n-2}\)).
Each reflection lies in a unique Klein four
$K_s = \{1,z,x_s,x_sz\}$, which is its unique cover.
All reflections are conjugate under \(a\), yielding one conjugacy class.
Each \(K_s\) similarly has a unique cover, the dihedral group \(D_8\)
generated by \(a^{2^{n-3}}\) and \(x_s\); again one class.
For \(k = 3,\dots,n-2\), the dihedral subgroups
$D_{2^k} = \langle a^{2^{n-k}},x_s\rangle$ form a chain, each with unique
cover \(D_{2^{k+1}}\) (or \(M_d\) when \(k=n-2\)).
All subgroups of a given order are conjugate, contributing one class per
order \(2^3,\dots,2^{n-2}\).
Counting orders \(2,4,8,\dots,2^{n-2}\) gives \(n-2\) classes from \(M_d\).

Turning to \(M_q\), a generalised quaternion group.
The elements of order four outside \(\langle a^2\rangle\) are
$y_s = a^{2s+1}b$; each \(\langle y_s\rangle\) is cyclic of order four and
has unique cover the quaternion group \(Q_8\) generated by \(a^{2^{n-3}}\)
and \(y_s\).
All such cyclic subgroups are conjugate, giving one class.
For \(t = 3,\dots,n-2\), the quaternion subgroups
$Q_{2^t} = \langle a^{2^{n-t}},y_s\rangle$ form a chain, each with unique
cover \(Q_{2^{t+1}}\) (or \(M_q\) when \(t=n-2\)), and all subgroups of a
given order are conjugate.
This yields \(1 + (n-4) = n-3\) classes from \(M_q\).

The three maximals are pairwise non-conjugate because their isomorphism types
differ; they contribute $3$ classes.
Therefore the number of conjugacy classes of proper meet irreducible subgroups
of $G$ equals $2n-2$, and hence $w(\SD_{2^n})=2n-2$, as claimed.
\end{proof}

\subsection{Frobenius groups \(\AGL(1,p^n)\)}

We now compute the width for the affine linear groups over finite fields.
Let \(p\) be prime and \(n\ge 1\), and set
\[
  G = \AGL(1,\F_{p^n}).
\]
Write \(V=\F_{p^n}^+\) for the additive group of the field and
\(H=\F_{p^n}^\times\) for its multiplicative group.
Then \(V\) is an elementary abelian \(p\)-group of order \(p^n\),
\(H\) is cyclic of order \(m=p^n-1\), and $G=V\rtimes H$,
where \(H\) acts on \(V\) by multiplication.
This is a Frobenius group with kernel \(V\) and complement \(H\).

Every subgroup of \(G\) can be written uniquely in the form \(A\rtimes D\),
where \(A\le V\) is an \(\F_p\)-subspace and \(D\le H\) normalises \(A\).
Conjugation by an element \(h\in H\) sends \((A,D)\) to \((hA,D)\), while
conjugation by \(V\) preserves the subgroup \(D\le H\).
In particular, subgroups with distinct \(D\)-parts are never conjugate.

\begin{theorem}\label{tfrob}
Let \(G=\AGL(1,\F_{p^n})\) with \(p\) prime and \(n\ge 1\).  Then
\[
  w(G) = \Omega(p^n-1)+\tau(n),
\]
where \(\Omega(m)\) denotes the total number of prime factors of \(m\)
counted with multiplicity, and \(\tau(n)\) denotes the number of positive
divisors of \(n\).
\end{theorem}

\begin{proof}
As before we first classify proper meet irreducible subgroups into three
families.
The complement \(H\) is maximal, hence meet irreducible; all complements are
conjugate, so this contributes one conjugacy class.

Next, consider subgroups of the form \(V\rtimes D\) with \(D<H\).
These are normal.
Their overgroups are \(V\rtimes D'\) with \(D'\ge D\).
Since \(H\) is cyclic, \(V\rtimes D\) has a unique minimal overgroup among
subgroups of the form \(V\rtimes D'\) precisely when \(m/|D|\) is a prime
power (i.e., there is exactly one divisor $e$ with $|D|<e\le m$).
The proper subgroups \(D\) with this property are exactly those of order
\(m/p^t\) for each prime power \(p^t\mid m\); there are \(\Omega(m)\) of
them.
Each yields a distinct conjugacy class.

Now let \(r\) be a proper divisor of \(n\), and let \(L\subseteq\F_{p^n}\)
be the unique subfield of order \(p^r\).
Then \(L^\times\le H\) is cyclic of order \(p^r-1\).
Choose an \(L\)-hyperplane \(A\subset V\), i.e., an \(L\)-subspace with
\(\dim_L A = n/r-1\).
The subgroup \(K_{r,A}=A\rtimes L^\times\) is well defined.
We claim \(K_{r,A}\) is meet irreducible: any proper overgroup must be either
\(V\rtimes L^\times\) or \(G\), and because \(L^\times\) is the full
stabiliser of \(A\), the unique maximal overgroup is \(G\).
For fixed \(r\), all \(K_{r,A}\) are conjugate under \(H\); different \(r\)
give distinct conjugacy classes.
The number of proper divisors of \(n\) is \(\tau(n)-1\), contributing that
many classes.

Any other proper subgroup \(A\rtimes D\) with \(A<V\) and \(D<H\) fails to
be meet irreducible.
Let \(L\) be the subfield generated by \(D\); then \(D\le L^\times\) and
\(A\) is an \(L\)-subspace.
If \(A\) is not a hyperplane over \(L\), it has at least two distinct
\(L\)-superspaces, yielding incomparable overgroups.
If \(A\) is a hyperplane but \(D\neq L^\times\), then \(A\rtimes D\) is
contained in both \(A\rtimes L^\times\) and \(V\rtimes D\), which are
incomparable.
If \(A=V\) we are in the \(V\rtimes D\) case already covered.
Subgroups of the form \(D\) alone (\(A=0\)) lie in both \(H\) and
\(V\rtimes D\), giving two covers unless \(D=H\) (which is maximal).

Summing the three families:
\[
  w(G) = 1 + \Omega(m) + (\tau(n)-1) = \Omega(m)+\tau(n) = \Omega(p^n-1)+\tau(n),
\]
as claimed.
\end{proof}

\begin{example}
Take \(p=2\), \(n=3\).
Then \(G=\AGL(1,8)\) has order \(56\), \(m=7\), \(\Omega(7)=1\),
\(\tau(3)=2\), so \(w(G)=3\).
The three classes are: \(H\cong C_7\); \(V\) itself (since \(7/1\) is
prime); and an \(\F_2\)-hyperplane \(A\) of order \(4\) (the subfield case
with \(r=1\)).
\end{example}

\subsection{Modular maximal-cyclic groups \(M_n(2)\)}

Throughout this subsection, fix an integer $n\ge 4$ and set
\[
  G = M_n(2) = \langle a,b\mid a^{2^{n-1}}=b^2=1,\;
    bab^{-1}=a^{1+2^{n-2}}\rangle.
\]
The cyclic subgroup $\langle a\rangle$ has order $2^{n-1}$, and conjugation
by $b$ raises $a$ to the power $1+2^{n-2}$, which is an automorphism of
order two.
The centre and Frattini subgroup of $G$ coincide:
\[
  Z(G)=\Phi(G)=\langle a^2\rangle\cong C_{2^{n-2}}.
\]
Conjugation by $a^r$ maps $b$ to $a^{r\cdot 2^{n-2}}b$, so two elements
$a^s b$ and $a^t b$ generate conjugate subgroups if and only if
$s\equiv t\pmod{2^{n-2}}$; this governs the conjugacy structure among
subgroups meeting the coset $\langle a\rangle b$.

\begin{theorem}\label{tmod}
For $n\ge 4$ one has $w(M_n(2))=2(n-1)$.
\end{theorem}

\begin{proof}
We classify proper meet-irreducible subgroups of
\[
  G = M_n(2) = \langle a,b\mid a^{2^{n-1}}=b^2=1,\;
    bab^{-1}=a^{1+2^{n-2}}\rangle
\]
into conjugacy classes directly.

Since $G/\Phi(G)\cong C_2\times C_2$, there are exactly three maximal
subgroups, all normal and of index two:
\[
  M_1=\langle a\rangle\cong C_{2^{n-1}},\quad
  M_2=\langle a^2,b\rangle\cong C_{2^{n-2}}\times C_2,\quad
  M_3=\langle ab\rangle\cong C_{2^{n-1}}.
\]
Each is meet-irreducible with unique cover $G$; being normal, each forms its
own conjugacy class, contributing three classes.
Any subgroup contained in two distinct maximals cannot be meet-irreducible.
Proper subgroups of the cyclic groups $M_1$ and $M_3$ each lie in the unique
cyclic subgroup of half their order, which itself lies in all three maximals,
so no proper subgroup of $M_1$ or $M_3$ qualifies.

It remains to examine $M_2\cong C_{2^{n-2}}\times C_2$.
For $1\le j\le n-3$, the subgroup lattice of $M_2$ at each level $j$ contains
exactly two meet-irreducible members: the subgroup
$\langle a^{2^j}, b\rangle\cong C_{2^{n-1-j}}\times C_2$, whose unique cover
inside $M_2$ is $\langle a^{2^{j-1}},b\rangle$, and the cyclic subgroup
$\langle a^{2^{j-1}}b\rangle$ of order $2^{n-j}$, whose unique cover is
likewise $\langle a^{2^{j-1}},b\rangle$.
The purely cyclic subgroup $\langle a^{2^j}\rangle$ at the same level is not
meet-irreducible, as it is covered by both $\langle a^{2^{j-1}}\rangle$ and
$\langle a^{2^j},b\rangle$.
All subgroups at a fixed level are normal, so each forms a singleton conjugacy
class; the $n-3$ levels together contribute $2(n-3)$ classes.

At the bottom level, $\langle a^{2^{n-2}}\rangle$ has three covers and is
not meet-irreducible.
The involutions $b$ and $a^{2^{n-2}}b$ each have
$\langle a^{2^{n-2}},b\rangle$ as their unique cover; since
$a\cdot\langle b\rangle\cdot a^{-1}=\langle a^{2^{n-2}}b\rangle$, they are
conjugate and form one conjugacy class, contributing a single class.

Summing over all families,
\[
  w(G) = 3 + 2(n-3) + 1 = 2n-2 = 2(n-1),
\]
as claimed.
\end{proof}

\section{Complexity of dihedral groups \texorpdfstring{$D_{p^n}$}{Dpn}}

The dihedral group $D_{p^n}$ presents a genuinely richer subgroup lattice
than the cyclic case $C_{p^n q}$.
While $\Sub(C_{p^n q})/C_{p^n q}$ is a grid, the quotient
$\Sub(D_{p^n})/D_{p^n}$ carries a conjugacy class of reflections
$d_0 = [\langle s\rangle]$ whose intersections with cyclic subgroups are
governed by the coprimality $\gcd(2,p)=1$.
This introduces a cross-restriction phenomenon
(Lemma~\ref{lem:anchor}) with no analogue in the $C_{p^n q}$ setting, and
its management requires a new bridge-exclusion argument
(Proposition~\ref{prop:bridge}).
Despite this additional complication in the proof, the final answer equals
that for $C_{p^n q}$.

\subsection{Setup and the rank function}

Fix an odd prime $p$ and an integer $n\ge 1$.  Set
\[
  G = D_{p^n} = \langle\, r, s \mid r^{p^n} = s^2 = e,\;
    srs^{-1} = r^{-1}\,\rangle.
\]
Every subgroup of $G$ belongs to one of two families:
\[
  C_{p^k} := \langle r^{p^{n-k}}\rangle,\quad 0\le k\le n,\qquad
  D_{p^k} := \langle r^{p^{n-k}},s\rangle,\quad 0\le k\le n,
\]
where $C_{p^0}=1$ and $D_{p^0}\cong C_2$ is generated by a reflection.
Since $p$ is odd, $C_{p^k}$ is the unique subgroup of order $p^k$ (hence
normal in $G$), and all copies of $D_{p^k}$ are conjugate in $G$.
We write $c_k$ and $d_k$ for the corresponding conjugacy classes.

The key arithmetic fact used throughout is:
\begin{equation}\label{eq:cop}
  \langle s\rangle \cap C_{p^k} = 1 \qquad\text{for all }k\ge 0,
\end{equation}
which holds because $|\langle s\rangle|=2$ and $|C_{p^k}|=p^k$ are coprime.

\begin{definition}\label{def:rank}
Define $P\colon\Sub(G)\to\{0,1,\dots,n+1\}$ by
\[
  P(C_{p^k})=k,\qquad P(D_{p^k})=k+1.
\]
This is the restriction to $\Sub(D_{p^n})$ of the rank map of
\cite[Definition~2.10]{ABB+25}.
\end{definition}

\begin{proposition}\label{prop:rankstrict}
$P$ is strictly order-preserving: $H\subsetneq K$ implies $P(H)<P(K)$.
\end{proposition}

\begin{proof}
We enumerate every possible type of proper inclusion.
\begin{itemize}
\item $C_{p^j}\subsetneq C_{p^k}$: occurs iff $j<k$, giving
  $P(C_{p^j})=j<k=P(C_{p^k})$.
\item $C_{p^j}\le D_{p^k}$: occurs iff $j\le k$; then
  $P(C_{p^j})=j\le k<k+1=P(D_{p^k})$.
\item $D_{p^j}\subsetneq D_{p^k}$: occurs iff $j<k$, giving
  $P(D_{p^j})=j+1<k+1=P(D_{p^k})$.
\item $D_{p^j}\le C_{p^k}$: impossible.  For $j\ge 1$, $D_{p^j}$ is
  non-abelian while $C_{p^k}$ is abelian.  For $j=0$, one would need
  $|D_{p^0}|=2$ to divide $|C_{p^k}|=p^k$, which fails since $p$ is odd.
  \qedhere
\end{itemize}
\end{proof}

By Proposition~\ref{prop:rankstrict}, $P$ extends to a well-defined map on
intervals: for $(H,K)\in I(\Sub(G))$ we write $P(H,K)=(P(H),P(K))$ and call
this the \emph{$P$-arc} of the arrow.

\begin{definition}\label{def:arrowtypes}
Non-trivial arrows in $I(\Sub(G))$, up to conjugacy, are classified as
follows:
\begin{itemize}
\item a \emph{$C$-arrow}: $(c_i,c_j)$ for $0\le i<j\le n$;
\item a \emph{$D$-arrow}: $(d_i,d_j)$ for $0\le i<j\le n$;
\item an \emph{$X$-arrow}: $(c_i,d_j)$ for $0\le i\le j\le n$
  (including vertical arrows $i=j$).
\end{itemize}
An arrow is \emph{left-anchored} if its source is $c_0$ or $d_0$.
For a minimal generating set $S$, we write $C$, $D$, $X$ for the counts of
$C$-, $D$-, and $X$-arrows in $S$ respectively.
\end{definition}

\begin{remark}
The three types of Definition~\ref{def:arrowtypes} cover all non-trivial
arrows, since Proposition~\ref{prop:rankstrict} rules out inclusions
$D_{p^j}\le C_{p^k}$.
\end{remark}

Figure~\ref{fig:arrowtypes} illustrates the three arrow types for $n=2$.

\begin{figure}[ht]
\centering
\begin{tikzpicture}[
  every node/.style = {mynode},
  scale = 1.1
]
\node (c0) at (0,0)   {$c_0$};
\node (c1) at (0,1.4) {$c_1$};
\node (c2) at (0,2.8) {$c_2$};
\node (d0) at (2,0)   {$d_0$};
\node (d1) at (2,1.4) {$d_1$};
\node (d2) at (2,2.8) {$d_2$};
\draw[hasseedge] (c0)--(c1); \draw[hasseedge] (c1)--(c2);
\draw[hasseedge] (d0)--(d1); \draw[hasseedge] (d1)--(d2);
\draw[hasseedge] (c0)--(d0); \draw[hasseedge] (c1)--(d1);
\draw[hasseedge] (c2)--(d2);
\draw[hasseedge] (c0) to[bend left=10] (d1);
\draw[hasseedge] (c1) to[bend left=10] (d2);
\draw[solidarrow, blue] (c0) to[bend left=35] (c2);
\draw[solidarrow, red]  (d0) to[bend right=35] (d2);
\draw[solidarrow, teal] (c0) to[bend left=10] (d1);
\end{tikzpicture}
\caption{The three arrow types of Definition~\ref{def:arrowtypes} in
$\Sub(D_{p^2})$.  The $C$-arrow $(c_0,c_2)$ is blue (left column), the
$D$-arrow $(d_0,d_2)$ is red (right column), and the $X$-arrow $(c_0,d_1)$
is teal (diagonal).  Thin grey lines are the Hasse edges.}
\label{fig:arrowtypes}
\end{figure}
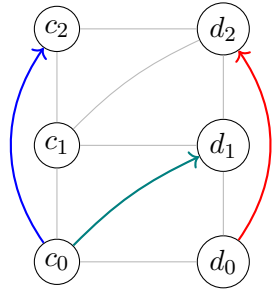

\subsection{Upper bound}

\begin{lemma}[Left-anchoring constraints]\label{lem:anchor}
Let $S$ be a minimal generating set for $T\in\cT(D_{p^n})$.
\begin{enumerate}[(i)]
\item At most one left-anchored arrow of each type belongs to $S$.
\item If $(d_0,d_k)\in S$, then $(c_0,c_j)\notin S$ for all $j\le k$.
  The same conclusion holds with $(d_0,d_k)$ replaced by any $X$-arrow
  $(c_0,d_k)\in S$.
\end{enumerate}
\end{lemma}

\begin{proof}
(i) We treat $C$-arrows; the other two types are identical.
Suppose $(c_0,c_j),(c_0,c_l)\in S$ with $j<l$.
Apply the restriction axiom to $(c_0,c_l)$ with $L=C_{p^j}\le C_{p^l}$:
\[
  (c_0\cap c_j,\; c_j) = (1\cap C_{p^j},\,C_{p^j}) = (c_0,c_j)\in\langle S\rangle.
\]
Hence $(c_0,c_j)\in\langle S\setminus\{(c_0,c_j)\}\rangle$, contradicting
minimality.

(ii) Suppose $(d_0,d_k)\in S$ and fix $j\le k$.
Since $C_{p^j}\le D_{p^k}$, apply the restriction axiom to $(d_0,d_k)$
with $L=C_{p^j}$:
\[
  (d_0\cap c_j,\;c_j)
  = \bigl(\langle s\rangle\cap C_{p^j},\,C_{p^j}\bigr)
  \stackrel{\eqref{eq:cop}}{=}
  (1,C_{p^j}) = (c_0,c_j)\in\langle S\rangle.
\]
Hence $(c_0,c_j)\notin S$ by minimality.
The case $(c_0,d_k)\in S$ is identical.
\end{proof}

Figure~\ref{fig:anchor} illustrates both parts of Lemma~\ref{lem:anchor}
for $n=2$.

\begin{figure}[ht]
\centering
\begin{tabular}{cccc}
\begin{tikzpicture}[
    every node/.style={mynode, minimum size=13pt, font=\scriptsize},
    scale=0.82]
  \node(c0)at(0,0){$c_0$};\node(c1)at(1.1,0){$c_1$};\node(c2)at(2.2,0){$c_2$};
  \node(d0)at(0,1.4){$d_0$};\node(d1)at(1.1,1.4){$d_1$};\node(d2)at(2.2,1.4){$d_2$};
  \draw[hasseedge](c0)--(c1);\draw[hasseedge](c1)--(c2);
  \draw[hasseedge](d0)--(d1);\draw[hasseedge](d1)--(d2);
  \draw[hasseedge](c0)--(d0);\draw[hasseedge](c1)--(d1);\draw[hasseedge](c2)--(d2);
  \draw[hasseedge](c0)to[bend left=12](d1);\draw[hasseedge](c1)to[bend left=12](d2);
  \draw[solidarrow,blue](c0)to[bend left=40](c2);
  \draw[dashedarrow,blue!60](c0)to[bend right=30](c1);
  \node[draw=none,fill=none,font=\small] at(1.1,-0.85){(i) C-type};
\end{tikzpicture}
&
\begin{tikzpicture}[
    every node/.style={mynode, minimum size=13pt, font=\scriptsize},
    scale=0.82]
  \node(c0)at(0,0){$c_0$};\node(c1)at(1.1,0){$c_1$};\node(c2)at(2.2,0){$c_2$};
  \node(d0)at(0,1.4){$d_0$};\node(d1)at(1.1,1.4){$d_1$};\node(d2)at(2.2,1.4){$d_2$};
  \draw[hasseedge](c0)--(c1);\draw[hasseedge](c1)--(c2);
  \draw[hasseedge](d0)--(d1);\draw[hasseedge](d1)--(d2);
  \draw[hasseedge](c0)--(d0);\draw[hasseedge](c1)--(d1);\draw[hasseedge](c2)--(d2);
  \draw[hasseedge](c0)to[bend left=12](d1);\draw[hasseedge](c1)to[bend left=12](d2);
  \draw[solidarrow,red](d0)to[bend right=40](d2);
  \draw[dashedarrow,red!60](d0)to[bend left=30](d1);
  \node[draw=none,fill=none,font=\small] at(1.1,-0.85){(ii) D-type};
\end{tikzpicture}
&
\begin{tikzpicture}[
    every node/.style={mynode, minimum size=13pt, font=\scriptsize},
    scale=0.82]
  \node(c0)at(0,0){$c_0$};\node(c1)at(1.1,0){$c_1$};\node(c2)at(2.2,0){$c_2$};
  \node(d0)at(0,1.4){$d_0$};\node(d1)at(1.1,1.4){$d_1$};\node(d2)at(2.2,1.4){$d_2$};
  \draw[hasseedge](c0)--(c1);\draw[hasseedge](c1)--(c2);
  \draw[hasseedge](d0)--(d1);\draw[hasseedge](d1)--(d2);
  \draw[hasseedge](c0)--(d0);\draw[hasseedge](c1)--(d1);\draw[hasseedge](c2)--(d2);
  \draw[hasseedge](c0)to[bend left=12](d1);\draw[hasseedge](c1)to[bend left=12](d2);
  \draw[solidarrow,teal](c0)to[bend left=15](d2);
  \draw[dashedarrow,teal!70](c0)to[bend left=8](d1);
  \node[draw=none,fill=none,font=\small] at(1.1,-0.85){(iii) X-type};
\end{tikzpicture}
&
\begin{tikzpicture}[
    every node/.style={mynode, minimum size=13pt, font=\scriptsize},
    scale=0.82]
  \node(c0)at(0,0){$c_0$};\node(c1)at(1.1,0){$c_1$};\node(c2)at(2.2,0){$c_2$};
  \node(d0)at(0,1.4){$d_0$};\node(d1)at(1.1,1.4){$d_1$};\node(d2)at(2.2,1.4){$d_2$};
  \draw[hasseedge](c0)--(c1);\draw[hasseedge](c1)--(c2);
  \draw[hasseedge](d0)--(d1);\draw[hasseedge](d1)--(d2);
  \draw[hasseedge](c0)--(d0);\draw[hasseedge](c1)--(d1);\draw[hasseedge](c2)--(d2);
  \draw[hasseedge](c0)to[bend left=12](d1);\draw[hasseedge](c1)to[bend left=12](d2);
  \draw[solidarrow,red](d0)to[bend right=40](d2);
  \draw[crossarrow](c0)to[bend right=30](c1);
  \draw[crossarrow](c0)to[bend left=40](c2);
  \node[draw=none,fill=none,font=\small] at(1.1,-0.85){(iv) Cross (new)};
\end{tikzpicture}
\end{tabular}
\smallskip

\noindent\footnotesize
\textbf{solid}: arrow in $S$ \quad
\textbf{- - -}: generated by restriction \quad
{\color{red!70!black}\textbf{- - -}}: cross-type forced (Lemma~\ref{lem:anchor}(ii))
\caption{Panels~(i)--(iii): in each type, the longer left-anchored solid
arrow generates the shorter dashed one by restriction, proving
Lemma~\ref{lem:anchor}(i).
Panel~(iv): the solid $D$-arrow $(d_0,d_2)$ forces both $C$-arrows
$(c_0,c_1)$ and $(c_0,c_2)$ (red dashed) into $\langle S\rangle$
via~\eqref{eq:cop}, proving Lemma~\ref{lem:anchor}(ii).
This cross-type phenomenon has no analogue in the $C_{p^n q}$ setting.}
\label{fig:anchor}
\end{figure}

\begin{lemma}[Inductive splitting]\label{lem:split}
Let $U\subseteq S$ be the arrows of any two specified types in $S$, and set
$U'=\{a\in U : a\text{ is not left-anchored}\}$.
Then $U'$ is a minimal generating set for $\langle U'\rangle$, viewed as a
transfer system on the sub-lattice $\Sub(G)_{\ge 1}\cong\Sub(D_{p^{n-1}})$.
\end{lemma}

\begin{proof}
The map $\iota\colon\Sub(D_{p^{n-1}})\to\Sub(D_{p^n})$ defined by
$c_k\mapsto c_{k+1}$ and $d_k\mapsto d_{k+1}$ is a lattice embedding
preserving all inclusion and conjugacy relations.
Under $\iota$, the arrows of $U'$ biject with arrows of the corresponding
types in $\Sub(D_{p^{n-1}})$.
Suppose for contradiction that some $a\in U'$ satisfies
$a\in\langle U'\setminus\{a\}\rangle$.
Since $U'\subseteq S$, we have
$\langle U'\setminus\{a\}\rangle\subseteq\langle S\setminus\{a\}\rangle$,
so $a\in\langle S\setminus\{a\}\rangle$, contradicting minimality of $S$.
\end{proof}

\begin{proposition}[Bridge-exclusion bounds]\label{prop:bridge}
Let $S$ be a minimal generating set for $T\in\cT(D_{p^n})$.  Then:
\begin{enumerate}[(1)]
\item $C+D\le n$, with equality implying
  $(c_0,c_n),(d_0,d_n)\in\langle S\rangle$.
\item $C+X\le n+1$, with equality implying $(c_0,d_n)\in\langle S\rangle$.
\item $D+X\le n+1$, with equality implying $(d_0,d_n)\in\langle S\rangle$.
\end{enumerate}
\end{proposition}

\begin{proof}
All three bounds are proved by induction on $n$.
In each part we let $U$ denote the arrows of the relevant pair of types in
$S$, set $U_0=\{a\in U: a\text{ is left-anchored}\}$, and put $U'=U\setminus
U_0$.
By Lemma~\ref{lem:split} and the inductive hypothesis, $|U'|$ is bounded by
the corresponding statement for $D_{p^{n-1}}$.

\smallskip
\noindent\emph{Base cases.}
When $n=0$, $G=D_1\cong C_2$ has the single non-trivial arrow $(c_0,d_0)$,
which is an $X$-arrow.
Bounds~(2) and~(3) hold trivially.
For~(1) at $n=1$: the only $C$-arrow is $(c_0,c_1)$ and the only $D$-arrow
is $(d_0,d_1)$.
By Lemma~\ref{lem:anchor}(ii) with $k=j=1$, if $(d_0,d_1)\in S$ then
$(c_0,c_1)\notin S$, giving $C+D\le 1=n$.

\smallskip
\noindent\emph{Proof of~(1): $C+D\le n$.}
Let $U$ be the $C$- and $D$-arrows of $S$.
By induction $|U'|\le n-1$.

\emph{Case~1:} $U_0$ contains no $D$-arrow.
By Lemma~\ref{lem:anchor}(i), $|U_0|\le 1$, so $C+D\le n$.

\emph{Case~2:} $U_0$ contains a $D$-arrow $(d_0,d_k)$ for some $1\le k\le n$.
By Lemma~\ref{lem:anchor}(ii), any $C$-arrow in $U_0$ must be $(c_0,c_l)$
with $l>k$, and by Lemma~\ref{lem:anchor}(i) there is at most one such
$C$-arrow.

\quad\emph{Sub-case~2a:} $U_0$ contains no $C$-arrow.
Then $|U_0|=1$ and $C+D\le n$ by $|U'|\le n-1$.

\quad\emph{Sub-case~2b:} $U_0$ contains exactly one $C$-arrow $(c_0,c_l)$
with $l>k$.
Then $|U_0|=2$.
We claim $|U'|\le n-2$, giving $C+D\le n$.
Define
\[
  U'_A := \{a\in U': \text{both endpoints of }a\text{ have }P\text{-value}\le k\},
\]
\[
  U'_B := \{a\in U': \text{both endpoints of }a\text{ have }P\text{-value}\ge l\}.
\]
We show $U'=U'_A\sqcup U'_B$ by ruling out bridges crossing the gap $(k,l)$.

\emph{No $D$-bridge.}
Suppose $(d_a,d_b)\in U'$ with $1\le a\le k$ and $b\ge l$.
Since $d_0=[\langle s\rangle]\le D_{p^a}$ for all $a$, restricting
$(d_0,d_k)$ to $d_a\le d_k$ gives $(d_0,d_a)\in\langle S\rangle$.
The composition axiom applied to $(d_0,d_a)$ and $(d_a,d_b)$ yields
$(d_0,d_b)\in\langle S\rangle$.
Restricting to $C_{p^l}\le D_{p^b}$ (valid since $l\le b$) and
applying~\eqref{eq:cop}:
\[
  \bigl(\langle s\rangle\cap C_{p^l},\,C_{p^l}\bigr)
  \stackrel{\eqref{eq:cop}}{=}
  (1,C_{p^l}) = (c_0,c_l)\in\langle S\rangle,
\]
contradicting $(c_0,c_l)\in S$.

\emph{No $C$-bridge.}
Suppose $(c_a,c_b)\in U'$ with $1\le a\le k$ and $b\ge l$.
Restricting $(d_0,d_k)$ to $C_{p^a}\le D_{p^k}$ and applying~\eqref{eq:cop}
gives $(c_0,c_a)\in\langle S\rangle$.
Composing with $(c_a,c_b)$ yields $(c_0,c_b)\in\langle S\rangle$.
Restricting to $C_{p^l}\le C_{p^b}$ gives $(c_0,c_l)\in\langle S\rangle$,
contradiction.

Thus $U'=U'_A\sqcup U'_B$.
Applying Lemma~\ref{lem:split} and induction to each block:
\[
  |U'| = |U'_A|+|U'_B| \le (k-1)+(n-l) \le (k-1)+(n-k-1) = n-2,
\]
where the last inequality uses $l\ge k+1$.

\emph{Equality.}
When $C+D=n$, the inductive hypothesis furnishes
$(c_1,c_n),(d_1,d_n)\in\langle U'\rangle$; composing with the left-anchored
arrows in $U_0$ yields $(c_0,c_n),(d_0,d_n)\in\langle S\rangle$.

\smallskip
\noindent\emph{Proof of~(2): $C+X\le n+1$.}
Let $U$ be the $C$- and $X$-arrows of $S$.
By induction $|U'|\le n$.
If $|U_0|\le 1$ then $C+X\le n+1$.
Suppose $|U_0|=2$: by Lemma~\ref{lem:anchor}(ii) the two left-anchored
arrows are an $X$-arrow $(c_0,d_k)$ and a $C$-arrow $(c_0,c_l)$ with $l>k$.
We eliminate bridges.
Suppose $(c_a,d_b)\in U'$ with $a\le k$ and $b\ge l$.
Restricting $(c_0,d_k)$ to $C_{p^a}\le D_{p^k}$ gives
$(c_0,c_a)\in\langle S\rangle$.
Composing with $(c_a,d_b)$ yields $(c_0,d_b)\in\langle S\rangle$.
Restricting to $C_{p^l}\le D_{p^b}$ and applying~\eqref{eq:cop} gives
$(c_0,c_l)\in\langle S\rangle$, contradiction.
The same argument rules out $C$-type bridges.
Hence $U'=U'_A\sqcup U'_B$ and $|U'|\le n-2$, giving $C+X\le n$.
\emph{Equality}: $C+X=n+1$ forces $|U_0|=1$ and $|U'|=n$; the inductive
hypothesis and composition yield $(c_0,d_n)\in\langle S\rangle$.

\smallskip
\noindent\emph{Proof of~(3): $D+X\le n+1$.}
Let $U$ be the $D$- and $X$-arrows of $S$.
By induction $|U'|\le n$.
A left-anchored $D$-arrow $(d_0,d_k)$ and a left-anchored $X$-arrow
$(c_0,d_m)$ do not make each other redundant: the transfer system generated
by $(c_0,d_m)$ alone consists of $X$- and $C$-type arrows only, so
$(d_0,d_k)\notin\langle\{(c_0,d_m)\}\rangle$; and restriction of $(d_0,d_k)$
produces only $D$- or $C$-type arrows, never an $X$-arrow starting at $c_0$,
so $(c_0,d_m)\notin\langle\{(d_0,d_k)\}\rangle$.
Hence $|U_0|\le 2$.
If $|U_0|\le 1$ we are done.
Suppose $|U_0|=2$ with $(d_0,d_k),(c_0,d_m)\in S$ and $k\le m$.
We eliminate bridges.

\emph{No $X$-bridge.}
Suppose $(c_a,d_b)\in U'$ with $a\le k$ and $b\ge m$.
Restricting $(d_0,d_k)$ to $C_{p^a}\le D_{p^k}$ and applying~\eqref{eq:cop}
gives $(c_0,c_a)\in\langle S\rangle$.
Composing with $(c_a,d_b)$ yields $(c_0,d_b)\in\langle S\rangle$.
Restricting to $D_{p^m}\le D_{p^b}$ and using~\eqref{eq:cop}:
\[
  (C_{p^0}\cap D_{p^m},\,D_{p^m})
  \stackrel{\eqref{eq:cop}}{=}
  (1,D_{p^m}) = (c_0,d_m)\in\langle S\rangle,
\]
contradicting minimality.

\emph{No $D$-bridge.}
Suppose $(d_a,d_b)\in U'$ with $1\le a\le k$ and $b\ge m$.
Restricting $(c_0,d_m)$ to $D_{p^a}\le D_{p^m}$ gives
$(c_0,d_a)\in\langle S\rangle$.
Composing with $(d_a,d_b)$ yields $(c_0,d_b)\in\langle S\rangle$.
Restricting to $D_{p^m}\le D_{p^b}$:
\[
  (C_{p^0}\cap D_{p^m},\,D_{p^m})
  \stackrel{\eqref{eq:cop}}{=}
  (1,D_{p^m}) = (c_0,d_m)\in\langle S\rangle,
\]
contradicting $(c_0,d_m)\in S$.

Thus $U'=U'_A\sqcup U'_B$ and
\[
  |U'| \le (k-1)+(n-m) \le n-1\quad(m\ge k),
\]
giving $D+X\le n+1$.
\end{proof}

Figure~\ref{fig:bridge} illustrates the bridge-exclusion argument from
Sub-case~2b of the proof of Proposition~\ref{prop:bridge}(1).

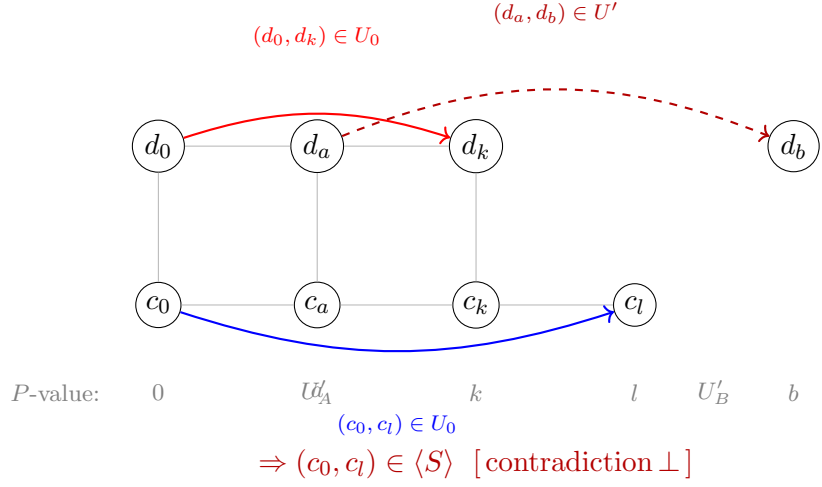
\begin{figure}[ht]
\centering
\begin{tikzpicture}[
  every node/.style={mynode},
  scale=1.05
]
\node (c0) at (0,0)   {$c_0$};
\node (ca) at (2,0)   {$c_a$};
\node (ck) at (4,0)   {$c_k$};
\node (cl) at (6,0)   {$c_l$};
\node (d0) at (0,2)   {$d_0$};
\node (da) at (2,2)   {$d_a$};
\node (dk) at (4,2)   {$d_k$};
\node (db) at (8,2)   {$d_b$};
\draw[hasseedge] (c0)--(ca);
\draw[hasseedge] (ca)--(ck);
\draw[hasseedge] (ck)--(cl);
\draw[hasseedge] (d0)--(da);
\draw[hasseedge] (da)--(dk);
\draw[hasseedge] (c0)--(d0);
\draw[hasseedge] (ca)--(da);
\draw[hasseedge] (ck)--(dk);
\draw[solidarrow, red]
  (d0) to[bend left=18]
  node[draw=none,fill=none,above,font=\tiny,yshift=2pt]
    {$(d_0,d_k)\in U_0$}
  (dk);
\draw[solidarrow, blue]
  (c0) to[bend right=18]
  node[draw=none,fill=none,below,font=\tiny,yshift=-2pt]
    {$(c_0,c_l)\in U_0$}
  (cl);
\draw[crossarrow]
  (da) to[bend left=22]
  node[draw=none,fill=none,above,font=\tiny,yshift=2pt]
    {$(d_a,d_b)\in U'$}
  (db);
\node[draw=none,fill=none,font=\scriptsize,gray] at (2,-1.1) {$U'_A$};
\node[draw=none,fill=none,font=\scriptsize,gray] at (7,-1.1) {$U'_B$};
\node[draw=none,fill=none,font=\scriptsize,gray] at (-1.3,-1.1)
    {$P\text{-value}$:};
\foreach \xpos/\val in {0/0, 2/a, 4/k, 6/l, 8/b}{
  \node[draw=none,fill=none,font=\scriptsize,gray] at (\xpos,-1.1) {$\val$};
}
\node[draw=none,fill=none,font=\small,red!70!black] at (4,-2.0)
    {$\Rightarrow (c_0,c_l)\in\langle S\rangle$ \;[\,contradiction\,$\bot$\,]};
\end{tikzpicture}
\caption{Bridge-exclusion argument of Sub-case~2b in
Proposition~\ref{prop:bridge}(1), for $D$-type bridges.
Both $(d_0,d_k)$ and $(c_0,c_l)$ lie in $U_0$ (solid), with $l>k$.
A hypothetical bridge $(d_a,d_b)\in U'$ (red dashed) leads to a
contradiction via three steps:
(i)~restrict $(d_0,d_k)$ to $d_a\le d_k$, giving $(d_0,d_a)\in\langle S\rangle$;
(ii)~compose with the bridge to obtain $(d_0,d_b)\in\langle S\rangle$;
(iii)~restrict to $c_l\le d_b$ and apply~\eqref{eq:cop} to derive
$(c_0,c_l)\in\langle S\rangle$ ($\bot$).
Since $(c_0,c_l)\in S$, this contradicts minimality.
Hence $U'$ splits as $U'_A\sqcup U'_B$ with no bridge crossing the gap
$(k,l)$.}
\label{fig:bridge}
\end{figure}

\begin{corollary}\label{cor:upper}
For any minimal generating set $S$ of $T\in\cT(D_{p^n})$,
\[
  |S| \le \left\lfloor\frac{3n}{2}\right\rfloor + 1.
\]
\end{corollary}

\begin{proof}
Since $|S|=C+D+X$, adding the three bounds:
\[
  2|S| = (C+D)+(C+X)+(D+X) \le n+(n+1)+(n+1) = 3n+2.
\]
As $|S|$ is an integer, $|S|\le\lfloor(3n+2)/2\rfloor=\lfloor 3n/2\rfloor+1$.
\end{proof}

\subsection{Lower bound via a partial rainbow}

We construct a partial rainbow on $\Sub(D_{p^n})$ achieving the upper bound
of Corollary~\ref{cor:upper}.

\begin{lemma}\label{lem:alpha}
For $0\le j<k\le n+1$, let $\alpha(j,k)$ denote the number of
$G$-conjugacy classes of non-trivial arrows in $I(\Sub(G))$ with $P$-arc
$(j,k)$.  Then
\[
  \alpha(j,k) = \begin{cases}
    1 & j=0,\;k=n+1,\\
    2 & j=0,\;1\le k\le n,\\
    2 & 1\le j\le n,\;k=n+1,\\
    3 & 1\le j<k\le n.
  \end{cases}
\]
\end{lemma}

\begin{proof}
A subgroup of $P$-value $j$ is $C_{p^j}$ (when $j\le n$) or $D_{p^{j-1}}$
(when $j\ge 1$).
Since $G$ acts transitively on conjugates of each $D_{p^l}$, the number of
conjugacy classes of arrows with a given $P$-arc equals the number of valid
(source type, target type) pairs, each contributing exactly one class.
We enumerate:
\begin{itemize}
\item[(a)] $C_{p^j}\hookrightarrow C_{p^k}$ (requires $j<k\le n$): both
  normal, conjugacy-fixed: $1$ class.
\item[(b)] $C_{p^j}\hookrightarrow D_{p^{k-1}}$ (requires $j\le k-1$,
  $k\le n$): since $C_{p^j}\unlhd G$, conjugation acts transitively on the
  target: $1$ class.
\item[(c)] $D_{p^{j-1}}\hookrightarrow D_{p^{k-1}}$ (requires $1\le j<k$,
  $k\le n$): $G$ acts transitively on conjugates of $D_{p^{j-1}}$, each
  inside the corresponding conjugate of $D_{p^{k-1}}$: $1$ class.
\item[(d)] $D_{p^{j-1}}\hookrightarrow C_{p^k}$: impossible by
  Proposition~\ref{prop:rankstrict}: $0$ classes.
\end{itemize}
Combining subject to boundary conditions:
$j=0,\,k=n+1$: unique inclusion $1\hookrightarrow G$, $\alpha=1$;
$j=0,\,1\le k\le n$: types (a) and (b) give $\alpha=2$;
$1\le j\le n,\,k=n+1$: types (b) and (c) give $\alpha=2$
  (type (a) would require $C_{p^{n+1}}$, which does not exist);
$1\le j<k\le n$: types (a), (b), (c) give $\alpha=3$.
\end{proof}

\begin{proposition}\label{prop:rainbow}
For each $n\ge 1$ there exists a partial rainbow on $\Sub(D_{p^n})$ of size
$\lfloor 3n/2\rfloor+1$.
\end{proposition}

\begin{proof}
Write $n=2m$ or $n=2m+1$, and consider the collection of $P$-arcs
\[
  \mathcal{R} = \{(0,2m+1),\,(1,2m),\,\dots,\,(m,m+1)\}.
\]
The $m+1$ arcs are strictly nested on $\{0,\dots,n+1\}$, so $\mathcal{R}$
is a rainbow.
Let $S$ be the set obtained by choosing, for each arc $(j,k)\in\mathcal{R}$,
exactly one arrow from each of the $\alpha(j,k)$ conjugacy classes given by
Lemma~\ref{lem:alpha}.
The conditions of \cite[Definition~3.2]{ABB+25} are satisfied by
construction; hence $S$ is a partial rainbow, and by
\cite[Proposition~3.3]{ABB+25} it is a minimal generating set for
$\langle S\rangle$, so $|S|=m(\langle S\rangle)\le c(D_{p^n})$.

\emph{Even case, $n=2m$.}
The arc $(0,n+1)$ has $\alpha=1$; each inner arc $(i,2m+1-i)$, $1\le i\le m$,
has $\alpha=3$.
Hence $|S|=1+3m=\lfloor 3n/2\rfloor+1$.

\emph{Odd case, $n=2m+1$.}
The arc $(0,n)=(0,2m+1)$ has $j=0$, $k=n$, giving $\alpha=2$.
Each inner arc $(i,2m+1-i)$, $1\le i\le m$, satisfies $1\le i<2m+1-i\le n-1$,
giving $\alpha=3$.
Hence $|S|=2+3m=\lfloor 3n/2\rfloor+1$.
\end{proof}

\subsection{Proof of Theorem D}

\begin{theorem}\label{tcomplex}
Let $p$ be an odd prime and $n\ge 1$.  Then
\[
  c(D_{p^n}) = \left\lfloor\frac{3n}{2}\right\rfloor+1
  = \begin{cases} 3k+1 & n=2k,\\ 3k+2 & n=2k+1.\end{cases}
\]
\end{theorem}

\begin{proof}
Corollary~\ref{cor:upper} gives $c(D_{p^n})\le\lfloor 3n/2\rfloor+1$.
Proposition~\ref{prop:rainbow} produces $T=\langle S\rangle\in\cT(D_{p^n})$
with $m(T)=\lfloor 3n/2\rfloor+1$, so $c(D_{p^n})\ge\lfloor 3n/2\rfloor+1$.
The parity formula is elementary arithmetic.
\end{proof}

\begin{remark}\label{rem:compare}
Theorem~\ref{tcomplex} is the direct non-abelian analogue of
\cite[Corollary~6.6]{ABB+25}.
The two results share identical arithmetic, but diverge structurally in the
proof of the upper bound.
For $C_{p^n q}$, top and bottom arrows are independent under restriction at
the left boundary, so each pairwise bound in \cite[Lemma~6.3]{ABB+25}
follows in a single induction.
For $D_{p^n}$, the cross-restriction of Lemma~\ref{lem:anchor}(ii) (a
consequence of $\gcd(2,p)=1$) entangles $C$- and $D$-arrows at the left
boundary, necessitating the bridge-exclusion argument in Sub-case~2b, which
has no counterpart in \cite{ABB+25}.

The lower bounds are more parallel.
In both cases the interior arcs (those with $1\le j<k\le n$) contribute
$\alpha=3$, and the boundary arcs contribute $1$ or $2$.
This coincidence of $\alpha$-values reflects the fact that $D_{p^n}$ and
$C_{p^n q}$ have the same count of conjugacy classes of proper inclusions
between non-extremal subgroups, forcing the two complexities to be equal.
\end{remark}

\section{Complexity of semidihedral groups \texorpdfstring{$\SD_{2^n}$}{SD2n}: a lower bound}

We now turn to a lower bound for the complexity of $\SD_{2^n}$.
Throughout this section fix $n\ge 4$ and $G=\SD_{2^n}$ with the notation of
Section~3.1.
Since $G$ is a $2$-group, the rank function of \cite[Definition~2.10]{ABB+25}
takes the form $P(H)=\log_2|H|$ for every $H\le G$, ranging over
$\{0,1,\dots,n\}$.

\begin{remark}\label{rem:sdnotation}
In this section, we write $\mathcal{C}_{2^r}$, $\mathcal{D}_{2^r}$,
$\mathcal{Q}_{2^r}$ for the conjugacy class of subgroups of $\SD_{2^n}$ that
are cyclic, dihedral-type (inside $M_d$), and quaternion-type (inside $M_q$)
of order $2^r$, respectively.
We abbreviate these as $C$-type, $D$-type, and $Q$-type subgroups at rank $r$.
This notation is distinct from the $D_{p^k}$ of Section~4, which denotes a
specific dihedral subgroup of $D_{p^n}$ of order $2p^k$.
\end{remark}

\subsection{Forbidden inclusions and arrow types}

The three-strand structure of $\Sub(G)$ noted in Section~3.1 imposes
constraints on which inclusions between subgroups are possible.

\begin{lemma}\label{lem:sdforb}
Let $H$ be a $D$-type subgroup of order $2^j$ and $K$ be a $Q$-type
subgroup of order $2^j$ in $\SD_{2^n}$, with $j\ge 1$ and $j\ge 2$
respectively.  Then:
\begin{enumerate}[(i)]
\item No $D$-type subgroup of order $2^j$ ($j\ge 1$) embeds in any cyclic
  subgroup;
\item No $Q$-type subgroup of order $2^j$ ($j\ge 2$) embeds in any cyclic
  subgroup;
\item No $Q$-type subgroup of order $2^j$ ($j\ge 2$) embeds in any $D$-type
  subgroup;
\item No $D$-type subgroup of order $2^j$ ($j\ge 1$) embeds in any $Q$-type
  subgroup.
\end{enumerate}
Consequently the only valid inclusion types between proper subgroups are
$C\hookrightarrow C$, $C\hookrightarrow D$, $C\hookrightarrow Q$,
$D\hookrightarrow D$, and $Q\hookrightarrow Q$.
\end{lemma}

\begin{proof}
All cyclic subgroups of $G$ lie in $M_c=\langle a\rangle$, which is abelian.
Hence no non-abelian subgroup embeds in any cyclic subgroup; this gives~(ii)
for $j\ge 3$ (since $Q$-type subgroups of order $\ge 8$ are non-abelian) and
all of~(i) for $j\ge 2$.

For~(i) with $j=1$: the $D$-type subgroup of order $2$ is $\langle x_s\rangle
=\langle a^{2s}b\rangle$ for some $s$, a reflection.  No reflection is a
power of $a$, since $b\notin M_c=\langle a\rangle$.  Hence no reflection
embeds in any cyclic subgroup. \checkmark

For~(ii) with $j=2$: the $Q$-type subgroup of order $4$ is $\langle y_s
\rangle=\langle a^{2s+1}b\rangle$, where $y_s^2=a^{(2s+1)2^{n-2}}\ne e$
for $n\ge 4$ (since $(2s+1)2^{n-2}\not\equiv 0\pmod{2^{n-1}}$ as $2s+1$
is odd).  So $y_s$ has order $4$, and $\langle y_s\rangle\cong C_4$ is
non-abelian-free but not contained in the cyclic $\langle a\rangle$ since
$y_s\notin M_c$. \checkmark

For~(iii): every element of a $Q$-type subgroup outside its cyclic part has
order $4$ (as computed above), while every element of a $D$-type subgroup
outside its cyclic part is a reflection of order $2$.  An element of order
$4$ cannot belong to a $D$-type subgroup. \checkmark

For~(iv): every $D$-type subgroup contains a reflection (an element of order
$2$ that is non-central, i.e., not equal to $z=a^{2^{n-2}}$), while in
every $Q$-type subgroup the unique element of order $2$ is $z$ (central in
$G$).  A non-central involution cannot embed in any $Q$-type subgroup.
\checkmark
\end{proof}

\subsection{The \texorpdfstring{$\alpha$}{alpha}-function for \texorpdfstring{$\SD_{2^n}$}{SD2n}}

For $0\le j<k\le n$, let $\alpha(j,k)$ denote the number of $G$-conjugacy
classes of non-trivial arrows in $\cI(\Sub(G))$ with $P$-arc $(j,k)$.
At each rank level $r$, the conjugacy classes of subgroups are:
rank $0$, one class ($C_1$); rank $1$, two classes (one $C$-type $\langle z\rangle$
and one $D$-type $\langle b\rangle$, both of order $2$);
ranks $2\le r\le n-1$, three classes ($C$-type, $D$-type, $Q$-type, all
of order $2^r$);
rank $n$, one class ($G$).
Since $G$ acts transitively on conjugates of each $D$-type and $Q$-type
subgroup, each valid (source type, target type) pair contributes exactly one
conjugacy class of arrows.

\begin{lemma}\label{lem:alphaSD}
\[
  \alpha(j,k) =
  \begin{cases}
    1 & j=0,\;k=n,\\
    2 & j=0,\;k=1,\\
    3 & j=0,\;2\le k\le n-1,\\
    2 & j=1,\;k=n,\\
    3 & 2\le j\le n-1,\;k=n,\\
    4 & j=1,\;2\le k\le n-1,\\
    5 & 2\le j<k\le n-1.
  \end{cases}
\]
\end{lemma}

\begin{proof}
We enumerate valid (source type, target type) pairs in each region.

For $j=0$, $k=n$: the unique inclusion $C_1\hookrightarrow G$ gives
$\alpha=1$.
For $j=0$, $k=1$: target types are one $C$-type and one $D$-type of rank
$1$, both reachable from $C_1$, giving $\alpha=2$.
For $j=0$, $2\le k\le n-1$: target types are $C$-type, $D$-type, and
$Q$-type of rank $k$, all reachable from $C_1$, giving $\alpha=3$.

For $j=1$, $k=n$: source types are the $C$-type and the $D$-type of rank
$1$, both of which embed in $G$, giving $\alpha=2$.
For $2\le j\le n-1$, $k=n$: source types are $C$-type, $D$-type, and
$Q$-type of rank $j$, all of which embed in $G$, giving $\alpha=3$.

For $j=1$, $2\le k\le n-1$: source types are $C$-type and $D$-type of rank
$1$; target types are $C$-type, $D$-type, and $Q$-type of rank $k$.
Lemma~\ref{lem:sdforb}(i) forbids $D\hookrightarrow C$ and
Lemma~\ref{lem:sdforb}(iv) forbids $D\hookrightarrow Q$.
The remaining valid pairs are $(C,C)$, $(C,D)$, $(C,Q)$, $(D,D)$,
giving $\alpha=4$.

For $2\le j<k\le n-1$: source types are $C$-type, $D$-type, and $Q$-type
of rank $j$; target types are $C$-type, $D$-type, and $Q$-type of rank $k$.
Lemma~\ref{lem:sdforb} forbids $D\hookrightarrow C$ (i), $Q\hookrightarrow C$
(ii), $Q\hookrightarrow D$ (iii), and $D\hookrightarrow Q$ (iv).
The remaining valid pairs are $(C,C)$, $(C,D)$, $(C,Q)$, $(D,D)$,
$(Q,Q)$, giving $\alpha=5$.
\end{proof}

\begin{remark}
The value $\alpha=5$ in the interior region $2\le j<k\le n-1$ reflects a
structural feature absent in the dihedral case $D_{p^n}$: the three-strand
structure of $\Sub(\SD_{2^n})$ with mutual non-embeddings between the $D$-
and $Q$-strands forces five valid inclusion types in place of three.
\end{remark}

\subsection{Lower bound via a partial rainbow}

\begin{proposition}\label{prop:rainbowSD}
For each $n\ge 4$ there exists a partial rainbow on $\Sub(\SD_{2^n})$ of
size $\lfloor 5(n-1)/2\rfloor$.
\end{proposition}

\begin{proof}
Write $n=2m+1$ (odd) or $n=2m$ (even).

\emph{Odd case, $n=2m+1$, so $m\ge 2$ since $n\ge 5$.}
Take the rainbow
\[
  R = \{(0,n),\,(1,n-1),\,(2,n-2),\,\dots,\,(m,m+1)\}
\]
of $m+1$ strictly nested arcs on $\{0,\dots,n\}$.
By Lemma~\ref{lem:alphaSD}:
the arc $(0,n)$ has $\alpha=1$;
the arc $(1,n-1)$ has $j=1$ and $k=n-1=2m\ge 2$ (since $m\ge 2$), giving
$\alpha=4$;
each arc $(i,n+1-i)$ for $2\le i\le m$ has $2\le i<n+1-i\le n-1$ (since
$i\le m$ implies $n+1-i\ge m+2>i$), giving $\alpha=5$.
There are $m-1$ such inner arcs.
Choosing one arrow from each conjugacy class per arc yields a partial rainbow
$S$ with
\[
  |S| = 1+4+5(m-1) = 5m = \left\lfloor\frac{5(n-1)}{2}\right\rfloor.
\]

\emph{Even case, $n=2m$, so $m\ge 2$ since $n\ge 4$.}
Take the rainbow
\[
  R_A = \{(1,n),\,(2,n-1),\,\dots,\,(m,m+1)\}
\]
of $m$ strictly nested arcs on $\{0,\dots,n\}$, with $0$ omitted.
By Lemma~\ref{lem:alphaSD}:
the arc $(1,n)$ has $j=1$, $k=n$, giving $\alpha=2$;
each arc $(i,2m+1-i)$ for $2\le i\le m$ has $j=i\ge 2$ and
$k=2m+1-i\le 2m-1=n-1$ (since $i\ge 2$) and $j<k$ (since $i\le m<2m+1-i$),
giving $\alpha=5$.
There are $m-1$ such inner arcs.
Hence
\[
  |S| = 2+5(m-1) = 5m-3 = \left\lfloor\frac{5(n-1)}{2}\right\rfloor.
\]

In both cases the $P$-arcs of $S$ form a rainbow by construction, and exactly
one arrow is chosen from each conjugacy class, so $S$ satisfies
\cite[Definition~3.2]{ABB+25}.
By \cite[Proposition~3.3]{ABB+25}, $S$ is a minimal generating set for
$\langle S\rangle$, giving $m(\langle S\rangle)=|S|\le c(\SD_{2^n})$.
\end{proof}

\begin{theorem}\label{tsdlower}
For $n\ge 4$,
\[
  c(\SD_{2^n}) \ge \left\lfloor\frac{5(n-1)}{2}\right\rfloor
  = \begin{cases} 5m-3 & n=2m,\\ 5m & n=2m+1.\end{cases}
\]
\end{theorem}

\begin{proof}
Immediate from Proposition~\ref{prop:rainbowSD} and
\cite[Proposition~3.3]{ABB+25}.
\end{proof}

The first few values of the lower bound are $c(\SD_{16})\ge 7$,
$c(\SD_{32})\ge 10$, $c(\SD_{64})\ge 12$, and $c(\SD_{128})\ge 15$,
corresponding to $n=4,5,6,7$ respectively.

\section{Concluding remarks}

The formulas obtained fit neatly into the pattern observed for dihedral and
quaternion groups in \cite{KLM+25}.
For \(\SD_{2^n}\) the value \(2n-2\) closely parallels \(w(D_{2^n})=2n-1\)
and \(w(Q_{2^{n+2}})=2n+2\); the slight discrepancy reflects the presence of
three rather than two maximal subgroups.
The quasidihedral structure, being a close cousin of the dihedral case,
simplifies the count via the same chain arguments used for \(M_d\) and
\(M_q\).
For \(\AGL(1,p^n)\) the expression \(\Omega(p^n-1)+\tau(n)\) exhibits an
attractive interplay between the arithmetic of \(p^n-1\) and the divisor
lattice of \(n\).

The complexity computation for $D_{p^n}$ (Theorem~\ref{tcomplex}) resolves
the question raised for this family: $c(D_{p^n})=\lfloor 3n/2\rfloor+1$,
matching $c(C_{p^n q})$ exactly despite the structurally richer subgroup
lattice.
The coincidence is explained by the equal $\alpha$-values for interior arcs
(Remark~\ref{rem:compare}).
Theorem~\ref{tsdlower} establishes that $c(\SD_{2^n})\ge\lfloor
5(n-1)/2\rfloor$, driven by the larger interior $\alpha$-value of $5$ arising
from the three-strand, no $D$-$Q$ interaction structure of $\Sub(\SD_{2^n})$.
The values $n=4,5,6,7$ have been verified against the exact complexity using
the \texttt{ninfty} software package of \cite{Bal25}, confirming equality in
each case.
We therefore conjecture that
\[
  c(\SD_{2^n}) = \left\lfloor\frac{5(n-1)}{2}\right\rfloor
\]
for all $n\ge 4$.
Establishing the matching upper bound would require pairwise bounds on the
five arrow-type counts analogous to Proposition~\ref{prop:bridge}, together
with bridge-exclusion arguments exploiting the behaviour of the central element
$z=a^{2^{n-2}}\in Z(\SD_{2^n})$ under the restriction axiom.
Computing the complexity for the affine groups $\AGL(1,p^n)$ and extending
the analysis to higher-dimensional affine groups $\AGL(d,q)$ remain
further open problems.


\begin{thebibliography}{99}

\bibitem{ABB+25} K. Adamyk, S. Balchin, M. Barrero, S. Scheirer, N. Wisdom,
  and V. Zapata Castro, \emph{On minimal bases in homotopical combinatorics},
  preprint (2025), arXiv:2506.11159.

\bibitem{Bal25} S. Balchin, \emph{ninfty: A software package for
  homotopical combinatorics}, arXiv:2504.01003, 2025.

\bibitem{BBR21} S. Balchin, D. Barnes, and C. Roitzheim,
  \emph{\(N_\infty\)-operads and associahedra},
  Pacific J. Math. \textbf{315} (2021), no.~2, 285--304.

\bibitem{BHK+25} L. Bao, C. Hazel, T. Karkos, A. Kessler, A. Nicolas,
  K. Ormsby, J. Park, C. Schleff, and S. Tilton,
  \emph{Transfer systems for rank two elementary abelian groups:
  characteristic functions and matchstick games},
  Tunis. J. Math. \textbf{7} (2025), no.~1, 167--191.

\bibitem{BH15} A. J. Blumberg and M. A. Hill,
  \emph{Operadic multiplications in equivariant spectra, norms, and
  transfers}, Adv. Math. \textbf{285} (2015), 658--708.

\bibitem{BMO25} S. Balchin, E. MacBrough, and K. Ormsby,
  \emph{The combinatorics of \(N_\infty\) operads for \(C_{q^{p^n}}\) and
  \(D_{p^n}\)}, Glasg. Math. J. \textbf{67} (2025), no.~1, 50--66.

\bibitem{BP21} P. Bonventre and L. A. Pereira,
  \emph{Genuine equivariant operads},
  Adv. Math. \textbf{381} (2021), Paper No.~107502.

\bibitem{KLM+25} S. Klanderman, C. Lewis, H. Monson, K. Shibata, and
  D. Van Niel, \emph{Characterizing transfer systems for non-abelian groups},
  preprint (2025).

\bibitem{Rub21a} J. Rubin, \emph{Combinatorial \(N_\infty\) operads},
  Algebr. Geom. Topol. \textbf{21} (2021), no.~7, 3513--3568.

\bibitem{Rub21b} J. Rubin, \emph{Detecting Steiner and linear isometries
  operads}, Glasg. Math. J. \textbf{63} (2021), no.~2, 307--342.

\end{thebibliography}
\end{document}